\documentclass[oneside,reqno]{amsart}
\usepackage{amssymb}
\usepackage{graphicx}
\usepackage{amssymb}
\usepackage[active]{srcltx}
\usepackage{a4wide}
\usepackage{amsmath}
\usepackage{amssymb}
\usepackage{amstext}
\usepackage{cite}
\usepackage{epsfig}
\usepackage{enumerate}
\usepackage{color}
\usepackage{amssymb}
\usepackage{float}
\everymath{\displaystyle}
\newtheorem{theorem}{Theorem}[section]
\newtheorem{lemma}[theorem]{Lemma}
\usepackage[linkcolor=blue, urlcolor=blue, citecolor=blue,
colorlinks, bookmarks]{hyperref}
\theoremstyle{definition}
\newtheorem{definition}[theorem]{Definition}

\newtheorem{example}[theorem]{Example}
\newtheorem{proposition}[theorem]{Proposition}
\newtheorem{note}[theorem]{Note}

\newtheorem{result}{Result}[section]

\theoremstyle{remark}
\newtheorem{remark}[theorem]{Remark}

\newcommand{\be}{\begin{equation}}
\newcommand{\ee}{\end{equation}}
\numberwithin{equation}{section}

\makeatletter
\newcommand{\setword}[2]{%
  \phantomsection
  #1\def\@currentlabel{\unexpanded{#1}}\label{#2}%
}
\makeatother

\usepackage{scalerel,stackengine}
\stackMath
\newcommand\reallywidehat[1]{%
\savestack{\tmpbox}{\stretchto{%
  \scaleto{%
    \scalerel*[\widthof{\ensuremath{#1}}]{\kern-.6pt\bigwedge\kern-.6pt}%
    {\rule[-\textheight/2]{1ex}{\textheight}}
  }{\textheight}%
}{0.5ex}}%
\stackon[1pt]{#1}{\tmpbox}%
}

\begin{document}

\title[Set-Valued $\alpha$-Fractal Functions]{Set-Valued $\alpha$-Fractal Functions}



\author{Megha Pandey}
\address{Department of Mathematical Sciences\\
Indian Institute of Technology (BHU)\\
Varanasi, India 221005}
\email{meghapandey1071996@gmail.com}
\author{Tanmoy Som}
\address{Department of Mathematical Sciences\\
Indian Institute of Technology (BHU)\\
Varanasi, India 221005}
\email{tsom.apm@iitbhu.ac.in}

\author{Saurabh Verma}
\address{Department of Applied Sciences\\
             Indian Institute of Information Technology Allahabad, Prayagraj, India 211015}
             \email{saurabhverma@iiita.ac.in}



\begin{abstract}
In this paper, we introduce the concept of the $\alpha$-fractal function and fractal approximation for a set-valued continuous map defined on a closed and bounded interval of real numbers. Also, we study some properties of such fractal functions. Further, we estimate the perturbation error between the given continuous function and its $\alpha$-fractal function. Additionally, we define a new graph of a set-valued function different from the standard graph introduced in the literature and establish some bounds on the fractal dimension of the newly defined graph of some special classes of set-valued functions. Also, we explain the need to define this new graph with examples. In the sequel, we prove that this new graph of an $\alpha$-fractal function is an attractor of an iterated function system. 
\end{abstract}

\maketitle
\section{Introduction} 
Approximation theory has gained appreciable attention in the literature. Fractal theory embraced the approximation theory in 1986 by Barnsley \cite{MF1} through his paper ``Fractal functions and interpolations''. Following this pioneering work of Barnsley, Navascu\'es \cite{navascues2005fractal, navascues2004generalization} and her group studied a parameterized class of fractal interpolation function, later known as $\alpha$-fractal function, associated with the continuous function defined on a real compact interval. After that, several theories have been developed concerning fractal interpolation functions. For example, in \cite{verma2020parameter} concept of $\alpha$-fractal function has been studied, and in \cite{chand2006generalized}, a generalized $C^r$ fractal interpolation function has been studied. 

In this paper, we have extended the notion of $\alpha$-fractal function in the case of set-valued maps. The significance of set-valued maps can be found in many essential areas, such as optimization theory, game theory, control theory, etc. One may refer \cite{Aubin} to understand the properties and nature of set-valued maps. The algebra of sets is different from those of numbers. There are different types of the sum has been given for sets. For instance, in \cite{artstein1989piecewise}, binary metric average of sets is defined. In \cite{dyn2005set}, Minkowski sum of two sets is used, and in \cite{berdysheva2019metric} the notion of the sum specified in \cite{artstein1989piecewise} has been extended, which is known as a metric linear combination of sets. In this paper, we have taken the Minkowski sum of sets.

Approximation of set-valued maps is one of the celebrated topics in the literature. Several theories have been given regarding the classical approximation of set-valued maps. See, for instance, \cite{levin1986multidimensional} where the notion of univariate data interpolation function has been given. Initially, approximation theory mainly focused on set-valued maps with convex images, known as convex set-valued maps. For example, in \cite{vitale1978approximation}, Vitale explored the approximation of convex set-valued maps with set-valued Bernstein polynomials. One may refer \cite{baier2011set,campiti2019korovkin,dyn2005set} for more research on the approximation of convex set-valued maps. The Minkowski sum of the two sets has been taken in all those researches. In \cite{artstein1989piecewise}, Arstein studied the approximation of set-valued maps having compact images, known as compact set-valued maps. Instead of Minkowski's sum of sets, he used the set of a sum of special pairs of elements, later known as  ``metric pairs". One may refer \cite{dyn2014approximation,berdysheva2019metric} for some more research on the approximation of compact set-valued maps. In this paper, we have studied the fractal approximation of set-valued maps.

Like fractal approximation, estimating the fractal dimension is also a fascinating area in fractal theory. It provides the statistical ratio of complexity with the details of how a fractal pattern changes with the scale it is measured. Several notions of fractal dimension have been introduced so far in the literature. For example, Hausdorff dimension , box dimension, packing dimension \cite{falconer2004fractal,barnsley2014fractals,massopust2016fractal}, etc. In this paper, we have worked on Hausdorff and box dimensions.

\subsection{Motivation and work done}

The concept of $\alpha$-fractal function, fractal approximation, and fractal dimension have been studied for different types of single-valued maps. For instance, in \cite{falconer2004fractal,barnsley2014fractals,jha2021dimensional} $\alpha$-fractal function for univariate single-valued maps and fractal dimension of the graph of some classes of univariate single-valued maps have been discussed. In \cite{verma2020parameter}, $\alpha$-fractal function and dimensional results for bivariate single-valued maps have been explored. In \cite{pandey2022some, Megha2022}, existence of $\alpha$-fractal function corresponding to continuous multivariate functions is given. Agrawal et.al.,\cite{vishal2021Lp} worked on the $L_p$-approximation using fractal functions on the Sierpi{\'n}ski gasket. Sahu and Priyadarshi \cite{sahu2020box} have studied the box dimension of the graph of a harmonic function on the Sierpi{\'n}ski gasket. Persuaded by these researches, we have extended the concept of classical approximation of set-valued maps to the fractal approximation of set-valued maps. We have introduced the notion of $\alpha$-fractal function for set-valued maps. Still, unlike in the case of single-valued maps, we noticed that, in general, set-valued $\alpha$-fractal function is not interpolatory in nature. Further, we have worked on estimating the fractal dimension of the graph of some special classes of set-valued maps.

\subsection{Delineation}

The proposed paper is assembled as follows. The next section is reserved for notations and preliminaries required for our study. Section \ref{sec3} is devoted to the development of fractal functions and to exploring their properties. In Section \ref{sec4}, we have studied the fractal approximations and constrained approximations of the set-valued map. Further, in Section \ref{sec5}, we have defined a new definition of the graph of set-valued maps and provided some dimensional results for this new graph. Also, we have explained the need to define this new graph. Moreover, we proved that there exists an iterated function system whose attractor is this new graph of the set-valued $\alpha$-fractal function. We have concluded our paper in Section \ref{conclusion}.


\section{Notations and Preliminaries}
The following are the notations which we have used in our paper :
\begin{itemize}
    \item $\mathbb{N}$: Collection of all natural numbers
    \item $\mathbb{R}$: Collection of all real numbers
    \item $I$: Closed and bounded interval of $\mathbb{R}$
    \item  $\mathcal{K}(\mathbb{R})=\{A\subset \mathbb{R}:A~\text{is a compact subset of}~\mathbb{R}\}$
    \item $\mathcal{K}_c(\mathbb{R})=\{A\in \mathcal{K}(\mathbb{R}):A~\text{is a convex subset of}~\mathbb{R}\}$
    \item $H_d(A,B)=\max \left\{\underset{a\in A}{\sup}\underset{b\in B}{\inf}\lvert a-b\rvert,\underset{b\in B}{\sup}\underset{a\in A}{\inf}\lvert b-a\rvert\right\}$ be the metric defined on $\mathcal{K}(\mathbb{R})$. It is broadly known as Hausdorff metric
    \item  $\mathcal{C}(I, \mathcal{K}(\mathbb{R}))$: Collection of all the continuous maps from $I$ to $\mathcal{K}(\mathbb{R})$
    \item $\mathcal{L}ip(I, \mathcal{K}(\mathbb{R}))$: Collection of all the Lipschitz maps from $I$ to $\mathcal{K}(\mathbb{R})$
    \item $\sigma\text{-}\mathcal{HC}$: Collection of all the H\"older continuous maps from $I$ to $\mathcal{K}(\mathbb{R})$ with exponent $\sigma$.
\end{itemize}
\begin{definition}\cite{falconer2004fractal}
Let $V\subseteq \mathbb{R}$ be a subset of $\mathbb{R}$, then the diameter of $V$ is defined as
\[\lvert V\rvert =\sup\left\{\lvert u-w \rvert: u,w \in V\right\}.\]
Let $F\subset \mathbb{R}$ be a subset of $\mathbb{R}$. A countable collection (or finite) of sets, $V_i$ is said to be $\eta$-cover of $F$ if it satisfies 
\[F \subseteq \bigcup_{i} V_i \text{ such that } \lvert V_i \rvert \leq \eta \text{ for all } i.\]
\end{definition}
\begin{definition}\cite{falconer2004fractal}
For a non-negative number $t$ and $\eta>0$, define 
\[H^t_{\eta}(F)= \inf \left\{\sum_{i=1}^{\infty}\lvert V_i\rvert^t:  \{V_i\} \text{ is a } \eta \text{-cover of } F \right\}.\]
Then, $t$-dimensional Hausdorff measure of $F$ is defined as 
\[H^t(F)=\lim_{\eta \rightarrow 0}H^t_{\eta}(F).\]
\end{definition}

\begin{definition}\cite{falconer2004fractal}
Consider $F \subseteq \mathbb{R}$ and $t \ge 0.$ The Hausdorff dimension of $F$ is defined as,
\[\dim_H(F)=\sup\{t:H^t(F)=\infty\big\}=\inf\big\{t:H^t(F)=0\}.\]
\end{definition}

\begin{definition}\cite{falconer2004fractal}
Assume $F$ be a non-empty and bounded subset of $\mathbb{R}$ and $N_{\eta}(F)$ be the lowest count of sets having at most $\eta$ diameter which can cover $F.$ The upper box dimension and lower box dimension of $F$, are defined as
\[\overline{\dim}_B(F)=\varlimsup_{\eta \rightarrow 0} \frac{\log N_{\eta}(F)}{- \log \eta} ~\text{ and }~ \underline{\dim}_B(F)=\varliminf_{\eta \rightarrow 0} \frac{\log N_{\eta}(F)}{- \log \eta}, \text{ respectively }.\]
If $\overline{\dim}_B(F)=\underline{\dim}_B(F)$, then it is called as box dimension of $F$, and it is defined as, $\dim_B(F)=\lim_{\substack{\eta \rightarrow 0}} \frac{\log N_{\eta}(F)}{- \log \eta}.$
\end{definition}

 \begin{definition}\cite{Aubin}
 Let $X$ and $Y$ be metric spaces and $F:X\rightrightarrows Y$ be a set-valued map from $X$ to $Y$, then the graph of $F$ is defined as,
 \begin{equation}\label{Gf1}
   G_F=\{(u,w) \in X \times Y: w \in F(u)\}.
 \end{equation}
 $F(u)$ is known as the image (or) the value of $F$ at $u$. If there is at least one element $u \in X$ such that $F(u)$ is non-empty, then $F$ is said to be nontrivial. If $F(u)$ is non-empty for each $u\in X$, then $F$ is known to be strict. The domain and range of $F$ is defined as
 \[ \text{Dom}(F):=\{u \in X: F(u) \ne \emptyset\} ~\text{and}~ \text{Im}(F):= \underset{u \in X}{\bigcup}F(u), \text{respectively}.\]
 \end{definition}
 
 \begin{definition}\label{order}
  Let $F, G:I\rightrightarrows \mathbb{R}$ be set-valued maps. Then, $F\leq G$ if and only if $F(u)\subseteq G(u)$ for all $u\in I$.
\end{definition}

 \begin{remark}
 If $F(u)$ is closed (convex, compact, bounded), then $F$ is said to be closed (convex, compact, bounded).
 \end{remark}
 \begin{definition}\cite{Aubin}
 Let $F: X \rightrightarrows Y $ be a set-valued mapping and $u \in \text{Dom}(F)$ such that for every neighborhood $\mathfrak{U}$ of $F(u),$
 \begin{equation}\label{uppersemi}
 \text{there exists } \eta > 0  \text{ such that } F(u') \subset \mathfrak{U} \text{ for all } u' \in B_X(u, \eta).
 \end{equation}
 Then, $F$ is characterized as upper semicontinuous at $u$. If it satisfies \eqref{uppersemi} for each $u\in \text{Dom}(F)$, then $F$ is known as upper semicontinuous function.
 
 If for every $w \in F(u)$ and every sequence $\{u_n\}\subset \text{Dom}(F)$ converges to $u$, there exists a sequence of elements $w_n\in F(u_n)$ converges to $w$, then $F$ is characterized as lower semicontinuous at $u$. If it is lower semicontinuous at each $u\in \text{Dom}(F)$, then $F$ is said to be the lower semicontinuous function.
 \end{definition}
 
\begin{lemma}\emph{\cite{Aubin}}
 A set-valued map $F$ is said to be convex if and only if $\text{for all } u_1, u_2 \in \text{Dom}(F),~\lambda \in [0,1]$, we have
 \[\lambda F(u_1)+(1-\lambda)F(u_2) \subset F\left(\lambda u_1+(1-\lambda)u_2\right).\]
\end{lemma}

\begin{definition}\cite{barnsley2014fractals}\label{Iterated}
Consider $(X,d)$ be a complete metric space and $\mathcal{K}(X)$ be the collection of all non-empty compact subsets of $X$ and $H_d$ be the Hausdorff metric defined on $\mathcal{K}(X)$ and $m$ be a positive integer such that $w_i: X\rightarrow X$ be contraction map for each $i\in \{1,\ldots,m\}$, then the system $\left\{(X,d): w_1,\ldots, w_m\right\}$ is known as Iterated Function System (IFS).
\end{definition}
 
 In the Definition \ref{Iterated}, IFS is satisfying the Banach contraction principle. One can construct the fractal by using those IFSs which satisfies some other contractions etc. For instance, in \cite{ri2018new} the construction of fractal using $\phi$-contraction has been introduced.
 
 
\begin{definition}\cite{bandt2006open}\label{OSC}
An IFS, $\big\{(X,d): \omega_1, \ldots \omega_k\big\}$ is said to satisfy Open Set Condition (OSC), if there exists a non-empty open set $V\subset \mathbb{R}$ such that 
\[\bigcup_{i=1}^{m}\omega_i(V)\subset V ~\text{ such that }~ \omega_i(V)\cap \omega_j(V)=\emptyset ~\text{ for }~ i\neq j.\]
Moreover, if $A$ is an attractor of IFS such that $V\cap A\neq \emptyset$, then the IFS is said to satisfy the Strong Open Set Condition (SOSC).
 \end{definition}

\section{Fractal functions in $C( I, \mathcal{K}( \mathbb{R}) )$}\label{sec3}
We know that space $\mathcal{C}( I, \mathcal{K}( \mathbb{R}) )$ when endowed with metric $d_{\mathcal{C}}$ is a complete metric space, where \[d_{\mathcal{C}}(F,G)=\lVert F-G \rVert _{\infty}=\sup_{\substack{u\in I}}H_d( F(u),G(u)).\] 

\begin{note}\cite{hu1997handbook}\label{new71}
   Recall some properties of the Hausdorff metric as follows.
\begin{enumerate}
   \item For $B, C, D, E \in  \mathcal{K} (X)$, we have
   \[H_d (B + C, D + E) \le  H_d (B,D) + H_d (C, E),\]
   where $Y + Z := \{y + z : y \in  Y\in \mathcal{K} (X) , z \in Z\in \mathcal{K} (X)\}$ is known as Minkowski sum of $Y$ and $Z.$
  \item For any $\lambda \in \mathbb{R}$, $H_d (\lambda B,\lambda D)= \lvert \lambda\rvert   H_d (B,D).$
\end{enumerate}
\end{note}

\begin{theorem}\label{Thm3.2}
Assume $F \in \mathcal{C}( I, \mathcal{K}( \mathbb{R})).$ Let   $\Delta:=\{(u_1,\ldots,u_N): u_1 < \cdots <u_N \}$ be a given data points such that it forms a partition of $I,$ and let $I_n=[u_{n},u_{n+1}].$ Let $L_n:I \rightarrow I_n $ be contractive homeomorphism such that $L_{n}(u_1)=u_{n} \text{ and } L_{n}(u_N)=u_{n+1}$ or $L_{n}(u_1)=u_{n+1} \text{ and } L_{n}(u_N)=u_{n}$. Further, assume that the base function $S \in \mathcal{C}( I, \mathcal{K}( \mathbb{R}) )$ satisfies
   \[S(u_1)-F(u_1)=S(u_N)-F(u_N),\]
   and scaling factor $\alpha \in \mathbb{R}.$ If $\lvert \alpha \rvert < 1,$ then there exists a unique function $F^{\alpha}_{\Delta,S} \in \mathcal{C}( I, \mathcal{K}( \mathbb{R}) )$ satisfying the following self-referential equation 
   \begin{equation}\label{self12}
      F^{\alpha}_{\Delta,S}(u)= F(u) + \alpha [F^{\alpha}_{\Delta,S}(L_n^{-1}(u)) - S(L_n^{-1}(u))] ~\text{ for every }~ u\in I_n,
   \end{equation}
    where $n \in J=\{1,\ldots, N-1\}. $ 
   \end{theorem}
   \begin{proof}
   Let $ \mathcal{C}_F(I,\mathcal{K}( \mathbb{R}))= \{ G \in \mathcal{C}(I,\mathcal{K}( \mathbb{R})): G(u_1)-S(u_1)~= G(u_N)-S(u_N)\}.$ It is elementary to observe that $\mathcal{C}_F(I,\mathcal{K}( \mathbb{R}))$ is a closed subset of $\mathcal{C}(I,\mathcal{K}( \mathbb{R})),$ hence $\left(\mathcal{C}_F(I,\mathcal{K}( \mathbb{R})), d_{\mathcal{C}}\right)$ is a complete metric space.
   Define \textit{Read-Bajraktarevi\'c }(RB) operator $\Phi: \mathcal{C}_F( I, \mathcal{K}( \mathbb{R}) ) \rightarrow \mathcal{C}_F( I, \mathcal{K}( \mathbb{R}) ) $ by 
   \[(\Phi G)(u)= F(u) + \alpha  [G(L_n^{-1}(u)) - S(L_n^{-1}(u))]\]
   for every $u \in I_n $ and $n \in J  .$ Well-definedness of $\Phi$ can be observed by using the assumptions we have taken for $F,S$, and $\alpha$. With the reference to Note \ref{new71}, we get 
   \begin{align*}
     &H_d((\Phi G)(u),(\Phi H)(u))\\
     =& H_d\Big(F(u) + \alpha  [G(L_n^{-1}(u)) - S(L_n^{-1}(u))],F(u) + \alpha  [H(L_n^{-1}(u))-S(L_n^{-1}(u))]\Big)\\
     \leq& H_d\Big( \alpha G(L_n^{-1}(u)),\alpha H(L_n^{-1}(u))\Big)\\
     =&\lvert \alpha \rvert H_d\Big(  G(L_n^{-1}(u)), H(L_n^{-1}(u))\Big)\\
     \le &  \lvert \alpha \rvert\sup_{u \in I} H_d( G(u),H(u))\\ 
     = &  \lvert \alpha \rvert\lVert G-H\rVert _{\infty}.  
   \end{align*}
   Since $ \lvert \alpha \rvert\lVert G-H\rVert _{\infty} $ is independent of $u,$ hence we have
   \begin{equation*}
    \lVert \Phi G-\Phi H\rVert _{\infty}
     \le   \lvert \alpha \rvert\lVert G-H\rVert _{\infty}.  
   \end{equation*}
  Because $\lvert \alpha \rvert < 1,$ $\Phi$ is a contraction on $\mathcal{C}( I, \mathcal{K}( \mathbb{R}) ).$ Hence, $\Phi$ has a fixed point in $\mathcal{C}(I,\mathcal{K}(\mathbb{R}))$. Let $F^{\alpha}_{\Delta,S}$ be that fixed point, then it satisfies the self-referential equation,
  \[F^{\alpha}_{\Delta,S}(u)= F(u) + \alpha  [F^{\alpha}_{\Delta,S}(L_n^{-1}(u)) - S(L_n^{-1}(u))]\]
   for every $u \in I_n$  and $n \in J.$
   \end{proof}
  
\begin{note}\label{unchange}
Throughout the paper, we denote $F^{\alpha}_{\Delta,S}$ as $F^{\alpha}$ if there is no ambiguity and $\Delta$, $S$, and $J$ has the same meaning it is in Theorem \ref{Thm3.2}.
\end{note}
  
\begin{remark}\label{rmrk3.3}
In the context of Equation \eqref{self12}, we get
\[F^{\alpha}(u_i)= F(u_i) + \alpha  F^{\alpha}(u_1) - \alpha S(u_1)=F(u_i) + \alpha  F^{\alpha}(u_N) - \alpha S(u_N) \text{ for every } u_i\in \Delta, \] 
where $i=1,\ldots, N$. Further, if $F^{\alpha}$ and $S$ are single-valued at the end points such that $F^{\alpha}(u_1) -  S(u_1)=F^{\alpha}(u_N) -  S(u_N)=\{0\},$ then $F^{\alpha}(u_i)=F(u_i)~$ for each $i=1,\ldots,N, \text{ this implies that }~F^{\alpha}$ is a set-valued fractal interpolation function.
\end{remark}

\begin{note}\label{Nt3.4}
The above remark hints at the following: in case $F$ and $S$ are single-valued at the end points such that $F(u_1)-S(u_1)=F(u_N)-S(u_N)=\{0\},$ then the set
\[\mathcal{C}_F(I,\mathcal{K}( \mathbb{R}))= \Big\{ G \in \mathcal{C}(I,\mathcal{K}( \mathbb{R})): G(u_1)-S(u_1)= G(u_N)-S(u_N)=\{0\}\Big\}\] is a complete metric space, and the RB operator $\Phi:\mathcal{C}_F(I,\mathcal{K}( \mathbb{R})) \to \mathcal{C}_F(I,\mathcal{K}( \mathbb{R}))$ as defined in Theorem \ref{Thm3.2} is well-defined and a contraction mapping. Therefore, we have a unique fixed point $F^{\alpha}$ of $\Phi$ satisfying $F^{\alpha}(u_i)=F(u_i)$ for all $i=1,\ldots,N,$ this shows that $F^{\alpha}$ is a set-valued fractal interpolation function.
\end{note}
    Here we give some examples of base functions $S \in \mathcal{C}( I, \mathcal{K}( \mathbb{R}) )$ satisfying $S(u_1)-F(u_1)=S(u_N)-F(u_N):$
    \begin{enumerate}[(i).]
        \item $S(u)=F(t(u))+(u-u_1)(F(u_1)-F(u_1)+(u_N -u)(F(u_N)-F(u_N))$, where $t:I \to I$ be a continuous function which satisfies $t(u_1)=u_1,~t(u_N)=u_N.$
        \item $S(u)=t(u)F(u)+(u-u_1)(F(u_1)-F(u_1)+(u_N -u)(F(u_N)-F(u_N))$, where $t:I \to \mathbb{R}$ be a continuous function which satisfies $t(u_1)=1 \text{ and } t(u_N)=1.$
    \end{enumerate}
   The H\"{o}lder space is defined as follows:
   \[\mathcal{HC}^{\sigma}(I, \mathcal{K}_c( \mathbb{R})  ) := \{G:I \rightarrow \mathcal{K}_c( \mathbb{R}): ~G\in \sigma\text{-}\mathcal{HC}\},\]
 
 Let us recall \cite{michta2020selection} that if we endow the space $\mathcal{HC}^{\sigma}(I,\mathcal{K}_c( \mathbb{R}))$ with metric
\[ H_{\sigma}^{(1)}(G,H)= \sup_{u \in I} H_d(G(u),H(u))+ \sup_{u,w \in I} \frac{H_d\big(G(u)+H(w),H(u)+G(w)\big)}{\lvert u-w \rvert^{\sigma}}.\]
 Then, by \cite[Proposition $1$]{michta2020selection}, it forms a complete metric space.
 \begin{note}
Take $L_n:I \rightarrow I_n $ as affine maps, such that $L_n(u)=a_nu +b_n$ for all $n \in J,$ where $a_n=\frac{u_{n+1}-u_n}{u_N-u_1}$ and $b_n=\frac{u_nu_N-u_1u_{n+1}}{u_N-u_1}$. 
 \end{note}
   
   \begin{theorem}\label{Thm3.5}
    Consider $F, S \in \mathcal{HC}^{\sigma}(I, \mathcal{K}_c( \mathbb{R})) $ such that $S(u_1)-F(u_1)=S(u_N)-F(u_N),$ and let $\alpha \in (-1,1).$ Then, $F^{\alpha}\in \sigma\text{-}\mathcal{HC}$ provided $\frac{\lvert \alpha \rvert}{a^{\sigma}}< 1$, where $a:= \min\{a_j: j \in J \}$. 
 \end{theorem}
 \begin{proof}
 Consider $ \mathcal{HC}^{\sigma}_F(I,\mathcal{K}_{c}(\mathbb{R}) )= \{ G \in \mathcal{HC}^{\sigma}(I,\mathcal{K}_{c}(\mathbb{R})): G(u_1)-S(u_1)=~G(u_N)-S(u_N) \}.$ It is easy to notice that $\mathcal{HC}^{\sigma}_F(I, \mathcal{K}_{c}( \mathbb{R}))$ is a closed subset of $\mathcal{HC}^{\sigma}(I,\mathcal{K}_{c}( \mathbb{R})),$ and hence complete with respect to the metric $H_{\sigma}^{(1)}$. Define a map $\Phi: \mathcal{HC}^{\sigma}_F(I,\mathcal{K}_c( \mathbb{R})) \rightarrow \mathcal{HC}^{\sigma}_F(I,\mathcal{K}_c( \mathbb{R}))$ as 
 \[(\Phi G)(u)=F(u)+\alpha ~(G-S)(L_j^{-1}(u))\]
 for each $u \in I_j $ where $j \in J.$ Clearly, $\Phi$ is well-defined. Now for $G, H \in \mathcal{HC}^{\sigma}_F(I,\mathcal{K}_c( \mathbb{R}) )$, we have
  \begin{align*}
    & H_{\sigma}^{(1)}(\Phi(G),\Phi(H))\\
    =&\sup_{u \in I}H_d(\Phi(G)(u),\Phi(H)(u))
    +\max_{j  \in J} \sup_{u \ne w, u,w  \in I_j} \frac{H_d\big(\Phi(G)(u)+\Phi(H)(w),\Phi(H)(u)+\Phi(G)(w)\big)}{\lvert u-w \rvert^{\sigma}}\\
    \le& ~\lvert \alpha \rvert \sup_{u \in I}H_d(G(u),H(u))\\
    &\hspace{2cm}+ \max_{j  \in J}  \sup_{u \ne w, u,w \in I_j} \frac{H_d\Big(\alpha G(L_j^{-1}(u)) +\alpha H(L_j^{-1}(w)),\alpha H(L_j^{-1}(u))+\alpha G(L_j^{-1}(w))\Big)}{\lvert u-w \rvert^{\sigma}}
\end{align*}
\begin{align*}
    &\le\lvert \alpha \rvert \sup_{u \in I}H_d(G(u),H(u))\\
   &\hspace{3cm} + \lvert \alpha \rvert \max_{j  \in J}  \sup_{u \ne w, u,w \in I_j} \frac{H_d\Big( G(L_j^{-1}(u)) + H(L_j^{-1}(w)), H(L_j^{-1}(u))+ G(L_j^{-1}(w))\Big)}{\lvert a_j\rvert^{\sigma}\lvert L_j^{-1}(u)-L_j^{-1}(w)\rvert^{\sigma}}\\ 
    &\le ~\lvert \alpha \rvert \sup_{u \in I}H_d(G(u),H(u))\\
    &\hspace{3cm}+ \frac{\lvert \alpha \rvert}{a^{\sigma}} \max_{j  \in J}  \sup_{u \ne w, u,w \in I_j} \frac{H_d\Big( G(L_j^{-1}(u)) + H(L_j^{-1}(w)), H(L_j^{-1}(u))+ G(L_j^{-1}(w))\Big)}{\lvert L_j^{-1}(u)-L_j^{-1}(w)\rvert^{\sigma}}\\
    &\le ~ \frac{\lvert \alpha \rvert}{a^{\sigma}}\left[ \sup_{u \in I}H_d(G(u),H(u))+    \sup_{u \ne w, u,w \in I} \frac{H_d\Big( G(u) + H(w), H(u)+ G(w)\Big)}{\lvert u-w \rvert^{\sigma}} \right]\\
    &\le \frac{\lvert \alpha \rvert}{a^{\sigma}} H_{\sigma}^{(1)}(G,H).
\end{align*}

Since $\frac{\lvert \alpha \rvert}{a^{\sigma}} < 1$, which implies $\Phi$ is a contraction map on $ \mathcal{HC}^{\sigma}_F(I,\mathcal{K}_c( \mathbb{R}) ).$ Now, the Banach contraction principle ensures that a unique fixed point of $\Phi$ exists. This completes the proof.
\end{proof}

\begin{definition}
Assume $ F:I \rightarrow \mathcal{K}( \mathbb{R})$ be a set-valued map. For every partition $ P:=\{(t_0, \ldots ,t_m): t_0< \dots <t_m \}$ of  $I,$  define 
\[V(F,I)= \sup_P \sum_{i=1}^{m} H_d(F(t_i),F(t_{i-1})),\]
where the supremum runs over all partitions, $P$ of $I.$\\ 
We set $\lVert F \rVert_{\mathcal{BV}}:= \lVert F \rVert_{\infty}+ V(F,I),$ where $\lVert F \rVert_{\infty}:=\underset{u \in I}{\sup}~\lVert F(u)\rVert =\underset{u \in I}{\sup}~ H_d(F(u),\{0\}).$ 
Then, $F$ will be characterize as a bounded variation function, if $\lVert F \rVert_{\mathcal{BV}} < \infty$. $\mathcal{BV}(I,\mathcal{K}( \mathbb{R}))$ will be denoted as a collection of all bounded variation function on $I$.
\end{definition}

\begin{remark}
It is interesting to write the following small observation :
define functions $F,T:[0,1] \rightarrow \mathcal{K}(\mathbb{R})$ as follows $F(x)=\begin{cases}\sin\left(\frac{1}{x}\right), &\text{when}~ x\neq 0\\
0, &\text{otherwise}\end{cases}$ and $T(x)=[-1,1].$ Here $F(x) \subset T(x)$ for each $x\in [0,1]$, such that $T \in \mathcal{BV}(I,\mathcal{K}(\mathbb{R}))$ while $F \notin \mathcal{BV}(I,\mathcal{K}(\mathbb{R}))$. This example shows that for set-valued mappings satisfying $F\le T$ does not imply $\lVert F \rVert_{\mathcal{BV}} \leq \lVert T\rVert _{\mathcal{BV}}$.
\end{remark}

As a prelude to our next result, we note the following lemma.
\begin{lemma}\label{lem2}
Consider $\{F_n\}$ is a sequence of set-valued continuous maps which uniformly converges to $F: I \to \mathcal{K}(\mathbb{R}).$ Then, for a given partition $P=\{(y_0,\dots,y_m):y_0  < \dots< y_m\}$ of $I$, we have 
\[ \sum_{i=1}^{m} H_d(F_n(y_i),F_n(y_{i-1})) \to \sum_{i=1}^{m} H_d(F(y_i),F(y_{i-1})).\]
Moreover,
\begin{equation}\label{eq3.2}
   \sup_{P}  \sum_{i=1}^{m} H_d(F(y_i),F(y_{i-1})) \le \liminf_{n \to \infty} \sup_{P} \sum_{i=1}^{m} H_d(F_n(y_i),F_n(y_{i-1})). 
\end{equation}
\end{lemma}
\begin{proof}
Let $P=\{(y_0,\ldots,y_m):y_0 < \cdots< y_m\}$ be a partition of $I$. The uniform convergence of $\{F_n\}$ implies
\begin{align*}
\lim_{n \to \infty} \sum_{i=1}^{m} H_d(F_n(y_i),F_n(y_{i-1})) =~ \sum_{i=1}^{m} H_d(F(y_i),F(y_{i-1})).
\end{align*}
Now for a given partition $P=\{(y_0,\ldots,y_m):y_0 <\cdots< y_m\}$ of $I$, we get 
\begin{align*}
\sum_{i=1}^{m} H_d(F(y_i),F(y_{i-1})) = &\sum_{i=1}^{m} H_d(\lim_{n \to \infty}F_n(y_i),\lim_{n \to \infty}F_n(y_{i-1})) \\ = &   \lim_{n \to \infty} \sum_{i=1}^{m} H_d(F_n(y_i),F_n(y_{i-1}))\\ \le  & 
\liminf_{n \to \infty} \sup_{P} \sum_{i=1}^{m} H_d(F_n(y_i),F_n(y_{i-1})),
\end{align*}

completing the proof.
\end{proof}

\begin{theorem}
The space $\Big(\mathcal{BV}(I,\mathcal{K}_c(\mathbb{R})), H_{\mathcal{BV}}\Big)$ is a complete metric space, where $$ H_{\mathcal{BV}}(G,H):= \lVert G-H\rVert _{\infty}+  \sup_{P}\sum_{i=1}^{m} H_d\Big(  G(y_i)+H (y_{i-1}), H(y_i) +G(y_{i-1})\Big).$$
\end{theorem}
\begin{proof}
Assume that $\{F_n\}$ is a Cauchy sequence in $\mathcal{BV}(I,\mathcal{K}_c( \mathbb{R}))$ with respect to $H_{\mathcal{BV}}.$ Equivalently, for $\epsilon >0$, there exists $n_0\in \mathbb{N}$ such that 
\[ H_{\mathcal{BV}}(F_n ,F_k)  < \epsilon \text{ for all } n,k \ge n_0.\]
Using the definition of $H_{\mathcal{BV}},$ we obtain $ \lVert F_n- F_k\rVert _{\infty}  < \epsilon \text{ for all } n,k \ge n_0.$ Since $(\mathcal{C}(I,\mathcal{K}_c( \mathbb{R})), \lVert .\rVert _{\infty})$ is a complete metric space, there exists a continuous function $F$ with $ \lVert F_n -F\rVert _{\infty} \to 0$ as $n \to \infty.$
We claim that $F \in \mathcal{BV}(I,\mathcal{K}_{c}( \mathbb{R}))$ and $H_{\mathcal{BV}}(F_n ,F)   \to 0$ as $n \to \infty.$ Let $P=\{(y_0,\dots,y_N):y_0 < \dots< y_N\}$ be a partition of $I$ and $n \ge n_0.$ From the reference to Lemma \ref{lem2}, we get
\begin{align*}
H_{\mathcal{BV}}(F_n,F)=&\lVert F_n -F\rVert _{\infty} +\sum_{i=1}^{m} H_d\Big(  F_n(y_i)+F (y_{i-1}), F(y_i) +F_n(y_{i-1})\Big) \\
=& \lim_{k \to \infty} \Bigg(\lVert F_n -F_k\rVert _{\infty} + \sum_{i=1}^{m} H_d\Big(  F_n(y_i)+F_k (y_{i-1}), F_k(y_i) +F_n(y_{i-1})\Big)\Bigg)\\  
\le &  \lim_{k \to \infty} \Bigg(\lVert F_n -F_k\rVert _{\infty} +\sup_{P} \sum_{i=1}^{m} H_d\Big(  F_n(y_i)+F_k (y_{i-1}), F_k(y_i) +F_n(y_{i-1})\Big)\Bigg)\\ 
\le &  \sup_{k \ge n_0} \Bigg(\lVert F_n -F_k\rVert _{\infty} +\sup_{P} \sum_{i=1}^{m} H_d\Big(  F_n(y_i)+F_k (y_{i-1}), F_k(y_i) +F_n(y_{i-1})\Big)\Bigg)\\
\leq &  \sup_{k \ge n_0} H_{\mathcal{BV}}(F_n,F_k)< \epsilon.
 \end{align*}
Since $P$ was arbitrary, therefore we have $H_{\mathcal{BV}}(F_n,F) < \epsilon ~ \text{ for all } ~n \ge n_0.$ \\
It remains to show that $F \in \mathcal{BV}(I, \mathcal{K}_{c}(\mathbb{R})).$
Now by using $H_d(B+D,C+D)=H_d(B,C)$ for every $B,C,D \in \mathcal{K}_c( \mathbb{R})$ (see for instance \cite{hu1997handbook}), we have
\begin{align*}
    &\sum_{i=1}^{m} H_d(F(y_i),F(y_{i-1}))\\
   =&\sum_{i=1}^{m} H_d(F(y_i)+F_n(y_{i-1}),F(y_{i-1})+F_n(y_{i-1}))\\
  \leq&\sum_{i=1}^{m}H_d(F(y_i)+F_n(y_{i-1}),F(y_{i-1})+F_n(y_{i}))+\sum_{i=1}^{m} H_d(F_n(y_i)+F(y_{i-1}),F(y_{i-1})+F_n(y_{i-1}))\\
  \leq&\sum_{i=1}^{m}H_d(F(y_i)+F_n(y_{i-1}),F(y_{i-1})+F_n(y_{i}))+\sum_{i=1}^{m} H_d(F_n(y_i),F_n(y_{i-1}))
  \leq H_{\mathcal{BV}}(F_n,F)+\lVert F_n \rVert_{\mathcal{BV}}.
\end{align*}
Since $H_{\mathcal{BV}}(F_n,F) < \epsilon$ and $F_n \in \mathcal{BV}(I,\mathcal{K}_c( \mathbb{R}))$, the above inequality yields that $F \in \mathcal{BV}(I, \mathcal{K}_c( \mathbb{R})).$
This completes the proof.
\end{proof}
\begin{theorem} \label{bounded2}
Consider $F \in \mathcal{BV}(I, \mathcal{K}_{c}( \mathbb{R}) )$, $\Delta$ as defined in Theorem \ref{Thm3.2}, $S \in \mathcal{BV}(I, \mathcal{K}_{c}( \mathbb{R}) )$ such that $S(u_1)-F(u_1)=~S(u_N)-F(u_N)$, and $\alpha \in (-1,1)$  with $\lvert \alpha \rvert< \frac{1}{N}.$ Then, $\alpha$-fractal function, $F^{\alpha}$ corresponding to $F$ is of bounded variation on $I$.
\end{theorem}
\begin{proof}
Consider $\mathcal{BV}_*(I, \mathcal{K}_{c}( \mathbb{R}) )= \{G \in \mathcal{BV}(I, \mathcal{K}_{c}( \mathbb{R}) ):G(u_1)-S(u_1)=~G(u_N)-S(u_N)\}.$ It is easy to prove that $\mathcal{BV}_*(I, \mathcal{K}_{c}( \mathbb{R}))$ is a closed subset of $\mathcal{BV}(I, \mathcal{K}_{c}( \mathbb{R})),$ hence complete with respect to metric $H_{\mathcal{BV}}$.
Define RB operator $\Phi: \mathcal{BV}_*(I, \mathcal{K}_{c}( \mathbb{R})) \rightarrow \mathcal{BV}_*(I, \mathcal{K}_{c}( \mathbb{R})) $ by 
\[(\Phi G)(u)= F(u) + \alpha  \big[G\big(L_j^{-1}(u)\big) - S\big(L_j^{-1}(u)\big)\big]\]
for each $u \in I_j$ and $j \in J.$ It is easy to observe the well-definedness of $\Phi$. Assume $ P=\{(t_0,\ldots ,t_m):t_0<\cdots <t_m\} $ is a partition of $I_j,$ where $m \in \mathbb{N}.$ We have

\begin{align*}
  &H_d\Big(\Phi(G)(t_i)+\Phi(H)(t_{i-1}),\Phi(H)(t_i)+\Phi(G)(t_{i-1})\Big)\\ \le & ~ H_d\Big( \alpha G\big(L_j^{-1}(t_i)\big)+\alpha H \big(L_j^{-1}(t_{i-1})\big), \alpha H\big(L_j^{-1}(t_i)\big) +\alpha G\big(L_j^{-1}(t_{i-1})\big)\Big)\\
  \le & ~ \lvert  \alpha\rvert  H_d\Big(  G\big(L_j^{-1}(t_i)\big)+H \big(L_j^{-1}(t_{i-1})\big), H\big(L_j^{-1}(t_i)\big) +G\big(L_j^{-1}(t_{i-1})\big)\Big).
 \end{align*}

 Summing over $i=1 $ to $m,$ we have
 \begin{align*}
      & \sum_{i=1}^{m} H_d\Big(\Phi(G)(t_i)+\Phi(H)(t_{i-1}),\Phi(H)(t_i)+\Phi(G)(t_{i-1})\Big) \\ 
      \le& \lvert \alpha \rvert \sum_{i=1}^{m} H_d\Big(  G\big(L_j^{-1}(t_i)\big)+H \big(L_j^{-1}(t_{i-1})\big), H\big(L_j^{-1}(t_i)\big) +G\big(L_j^{-1}(t_{i-1})\big)\Big) \\
      \le &\lvert \alpha \rvert \sup_{P}\sum_{i=1}^{m} H_d\Big(  G(t_i)+H (t_{i-1}), H(t_i) +G(t_{i-1})\Big).
 \end{align*}
 The above inequality is true for any partition of $I_j$. Hence, we get

 \begin{align*}
 &H_{\mathcal{BV}}(\Phi(G),\Phi(H)) \\
 =& \sup_{u \in I}H_d(\Phi(G)(u),\Phi(H)(u))+ \sup_{P}\sum_{i=1}^{m} H_d\Big(  \Phi G(t_i)+\Phi H (t_{i-1}), \Phi H(t_i) +\Phi G(t_{i-1})\Big)\\ 
 \le& \lvert \alpha \rvert \sup_{u \in I}H_d(G(u),H(u))+ N \lvert \alpha \rvert \sup_{P}\sum_{i=1}^{m} H_d\Big(  G(t_i)+H (t_{i-1}), H(t_i) +G(t_{i-1})\Big)\\ 
 \le & N \lvert \alpha \rvert H_{\mathcal{BV}}(G,H).
 \end{align*} 
 
As $\lvert \alpha \rvert < \frac{1}{N},$  $\Phi$ is a contraction map. Then, Banach fixed point theorem ensures that $\Phi$ has a unique fixed point, say $F^{\alpha}$. Further, this fixed point will satisfy the following self-referential equation, 
\[ F^{\alpha}(u)= F(u) + \alpha  \big[F^{\alpha}\big(L_j^{-1}(u)\big) - S \big(L_j^{-1}(u)\big)\big]~ \text{for each}~   u \in I_j, ~\text{where}~ j \in J.\]
 \end{proof}

Notice that function $F^{\alpha}$ is a parametric function depending on parameters, base function $S$, scaling function $\alpha$, partition $\Delta$ and the function $F$ itself. To observe collective behavior of $F^{\alpha}$ depending on some such parameters we define a set-valued map, $\mathcal{F}^{\alpha}_{S}: \mathcal{C}(I, \mathcal{K}(\mathbb{R})) \rightarrow \mathcal{C}(I, \mathcal{K}(\mathbb{R}))$ such that 
\begin{equation}\label{fractope}
    \mathcal{F}^{\alpha}_{S}(F)=F^{\alpha},~ \text{where}~ \alpha \in (-1,1).
\end{equation}
This map is known as fractal operator.
  
\begin{theorem}\label{fractoperator}
$\mathcal{F}^{\alpha}_{S}$ defined in \eqref{fractope} is a continuous map.
\end{theorem}
\begin{proof}
Let $\{F_{k}\}$ be a sequence in $\mathcal{C}(I, \mathcal{K}(\mathbb{R}))$ such that $F_{k}\to F$, then to prove $\mathcal{F}^{\alpha}_{S}$ is a continuous function, it is sufficient to prove that $F^{\alpha}_{n} \to F_{\alpha}$. Since $F_{n}\to F$, then for each $\epsilon$ there exists $n_0\in \mathbb{N}$ such that,
\begin{align*}
    \lVert F_n-F\rVert _{\infty} &< \epsilon (1-\lvert \alpha \rvert)~ \text{for all}~ n\geq n_0,\\
\text{equivalently, }\underset{u\in I}{\sup} H_{d}(F_n(u),F(u)) &< \epsilon (1-\lvert \alpha \rvert)  \text{ for all } n\geq n_0.
\end{align*}
Now, we have
\begin{align*}
&H_{d}(F^{\alpha}_{n}(x),F^{\alpha}(x))\\
=&H_{d}\bigg(F_{n}(x)+\alpha\left[F^{\alpha}_{n}(L^{-1}_{j}(x))-S(L^{-1}_{j}(x))\right], F(x)+\alpha\left[F^{\alpha}(L^{-1}_{j}(x))-S(L^{-1}_{j}(x))\right]\bigg)\\
\leq & H_{d}(F_{n}(x),F(x))+\lvert \alpha \rvert H_d(F^{\alpha}_{n}(L^{-1}_{j}(x)),F^{\alpha}(L^{-1}_{j}(x))).
\end{align*}
This implies,
\begin{align*}
\underset{x\in I}{\sup}~H_d(F^{\alpha}_{n}(x), F^{\alpha}(x)) &\leq \frac{1}{1-\lvert \alpha \rvert}~\underset{x\in I}{\sup}~H_d(F_{n}(x),F(x)),\\
\text{that is }\lVert F^{\alpha}_{n}(x)-F^{\alpha}(x)\rVert_{\infty} &< \epsilon ~\text{for all } n\geq n_0.
\end{align*}
This completes the proof.
\end{proof}

\begin{theorem}
For a fixed partition $\Delta$, the mapping $\mathcal{T}^{\Delta}_{S}:\mathcal{C}(I, \mathcal{K}(\mathbb{R})) \rightrightarrows \mathcal{C}(I, \mathcal{K}(\mathbb{R}))$ defined as,
\[\mathcal{T}^{\Delta}_{S}(F)=\left\{F^{\alpha}: \alpha \in (-1,1)\right\}\] 
is lower semi-continuous.
\end{theorem}  
  
\begin{proof}
Let $F \in \mathcal{C}(I, \mathcal{K}(\mathbb{R}))$ and let $F^{\alpha}\in \mathcal{T}^{\Delta}_{S}(F)$ and a sequence $F_{k}\in \mathcal{C}(I, \mathcal{K}(\mathbb{R}))$ such that $F_{k} \to F$. Using Theorem \ref{fractoperator}, we have $F^{\alpha}_{k}\to F^{\alpha}$, then clearly $F^{\alpha}_{k}\in \mathcal{T}^{\Delta}_{S}(F_k)$, establishing the result.
\end{proof}

\section{Approximation of set-valued functions}\label{sec4}
 In Section \ref{sec3}, we observe that $F^{\alpha}$ satisfies the following self-referential equation:
 \[F^{\alpha}(u)= F(u) + \alpha  \left[F^{\alpha}\left(L_j^{-1}(u)\right) - S\left(L_j^{-1}(u)\right)\right]\]
 for every $ u \in I_j$, where $j \in J$.\\
 Let us note the next result as a prelude.
 \begin{proposition}\label{new79}
 Between $F$ and $F^{\alpha}$, the following perturbation error will be obtained:
   \[\lVert F^{\alpha} -F \rVert _{\infty} \leq \frac{\lvert \alpha \rvert}{1- \lvert \alpha \rvert} \lVert F-S \rVert _{\infty}+\frac{2 \lvert \alpha \rvert}{1- \lvert \alpha \rvert} \lVert F \rVert_{\infty}.\]
 \end{proposition}
 \begin{proof}
 Using the self-referential equation and Note \ref{new71}, we get
\begin{align*}
   & H_d(F^{\alpha}(u), F(u))\\
   =& H_d\bigg(F(u) + \alpha  \left[F^{\alpha}\left(L_j^{-1}(u)\right) -                 S\left(L_j^{-1}(u)\right)\right], ~F(u)\bigg)\\ 
   \leq &  H_d\bigg( \alpha  \left[F^{\alpha}\left(L_j^{-1}(u)\right) -        S\left(L_j^{-1}(u)\right)\right], ~\{0\}\bigg)=\lvert \alpha \rvert H_d\bigg(F^{\alpha}\left(L_j^{-1}(u)\right)- S\left(L_j^{-1}(u)\right), ~\{0\}\bigg)\\
   \end{align*}
   \begin{align*}
    \leq&  \lvert \alpha \rvert H_d\bigg(F^{\alpha}\left(L_j^{-1}(u)\right) -  S\left(L_j^{-1}(u)\right),~ F\left(L_j^{-1}(u)\right)-F\left(L_j^{-1}(u)\right)\bigg)\\
    &\hspace{7cm}+\lvert \alpha \rvert H_d\bigg(F\left(L_j^{-1}(u)\right)-F\left(L_j^{-1}(u)\right), \{0\}\bigg)\\
    \leq &\lvert \alpha \rvert H_d\bigg(F^{\alpha}\left(L_j^{-1}(u)\right),  F\left(L_j^{-1}(u)\right)\bigg)+\lvert \alpha \rvert H_d\bigg(- S\left(L_j^{-1}(u)\right),-F\left(L_j^{-1}(u)\right)\bigg)\\
    &\hspace{7cm}+2 \lvert \alpha \rvert H_d\bigg(F\left(L_j^{-1}(u)\right),\{0\}\bigg)\\
     \leq& \lvert \alpha \rvert \sup_{u\in I_j, j \in J}H_d\bigg(   F^{\alpha}\left(L_j^{-1}(u)\right),  F\left(L_j^{-1}(u)\right)\bigg)+ \lvert \alpha \rvert \sup_{u\in I_j, j \in J}
     H_d\bigg(S\left(L_j^{-1}(u)\right),F\left(L_j^{-1}(u)\right)\bigg)\\
     &+2 \lvert \alpha \rvert \sup_{u \in I_j, j \in J}H_d\left(   F\left(L_j^{-1}(u)\right),\{0\}\right)\\
     \leq& \lvert \alpha \rvert \lVert F^{\alpha}-F \rVert _{\infty}+ \lvert \alpha \rvert \lVert F-S \rVert _{\infty}+2 \lvert \alpha \rvert  \lVert F \rVert_{\infty}.
   \end{align*}
 This in turn yields $\lVert F^{\alpha}-F \rVert _{\infty}\le  \lvert \alpha \rvert \lVert F^{\alpha}-F \rVert _{\infty}+ \lvert \alpha \rvert \lVert F-S \rVert _{\infty}+2 \lvert \alpha \rvert  \lVert F \rVert_{\infty}$. This establishes the proof.
 \end{proof}
 
   \begin{theorem} \label{BFOTHM5}
Consider $F \in \mathcal{C}(I , \mathcal{K}_{c}( \mathbb{R})).$ 
 For any $\epsilon>0$, there is a set-valued fractal polynomial $P^{\alpha}$  such that
 \[ \lVert F-P^{\alpha}\rVert_{\infty} <\epsilon .\]
 \end{theorem}
 \begin{proof}
 For $\epsilon >0$ using \cite{vitale1978approximation}, there is a set-valued polynomial function $P$ such that
 \[\lVert F-P \rVert _{\infty} <\frac{\epsilon}{3} .\]
 Choose a partition $\Delta_P=\{y_0,\ldots, y_M\}$ of $I $ and  a continuous function $S_P $ satisfying $S_P(y_0)-P(y_0)=S_P(y_M)-P(y_M)$, and $\alpha \in (-1,1)$ such that 
 \[\lvert \alpha \rvert <\min\left\{\frac{\frac{\epsilon}{3}}{\frac{\epsilon}{3}+\lVert P-S_P\rVert _{\infty}} ,\frac{\frac{\epsilon}{3}}{\frac{\epsilon}{3}+2\lVert P\rVert _{\infty}} \right\}.\]
 Then, we get
   \begin{align*}
    \lVert F-P^{\alpha} \rVert_{\infty} & \le \lVert F-P  \rVert _{\infty}+\lVert P- P^{\alpha}\rVert _{\infty}~ \text{ (using triangle inequality)}\\
    & \le \lVert F-P  \rVert _{\infty}+ \frac{\lvert \alpha \rvert}{1-\lvert \alpha \rvert}\lVert P- S_P\rVert _{\infty}+\frac{2\lvert \alpha \rvert}{1-\lvert \alpha \rvert}\lVert P\rVert _{\infty}~ \text{ (using Proposition \ref{new79})}\\
    & < \frac{\epsilon}{3} + \frac{\epsilon}{3}+ \frac{\epsilon}{3} \\
     & = \epsilon.
   \end{align*}
   
 \end{proof}
 \begin{remark}
 We took $\alpha \in \mathbb{R}$ in the above proof, such that
 \[\lvert \alpha \rvert <\min\left\{\frac{\frac{\epsilon}{3}}{\frac{\epsilon}{3}+\lVert P-S_P\rVert _{\infty}} ,\frac{\frac{\epsilon}{3}}{\frac{\epsilon}{3}+2\lVert P\rVert _{\infty}}  \right\}.\] 
 In this situation, $\alpha$ may be ``close" to $0$ and hence $P^\alpha$ may not be self-referential and it may behave as a classical  polynomial. In alter, if we fix $\alpha \in (-1,1)$ such that $\lvert \alpha \rvert < 1$, but otherwise arbitrary and choose a polynomial $P$ and 
a function $S_P \in \mathcal{C}(I, \mathcal{K}_c( \mathbb{R}))$ satisfying $S_P(u_1)-P(u_1)=S(u_N)-P(u_N)$ and
\[\lVert P-S_P\rVert   <\frac{(1-\lvert \alpha \rvert)\epsilon}{3\lvert \alpha \rvert } \text{ and }  \lVert P\rVert _{\infty}< \frac{(1-\lvert \alpha \rvert)\epsilon}{6 \lvert \alpha \rvert } .\]
This forces $F$ to be a zero set function. Hence, the analogue of \cite[Remark 5.2]{verma2020fractal} cannot be established in the setting of set-valued mappings. In particular, the recently developed notion of Bernstein fractal functions will not be useful in the approximation of set-valued functions.

 \end{remark}
With the reference to Theorem \ref{BFOTHM5}, we have
\begin{theorem}
 The set of set-valued fractal polynomials with a non-zero scale vector is dense in $\mathcal{C}(I , \mathcal{K}_c( \mathbb{R})).$
 \end{theorem}

\subsection{Constrained approximation}   
Here we target to study some constrained approximation aspects of fractal functions. Before proving the next theorem, let us recall a result and prove a lemma as a prelude.
\begin{result}\label{res 4.5}
Consider $X,~Y$ are topological spaces, $f:X\to Y$ is a continuous function and $S$ is a dense subset in $X$. If $f(u)\leq 0~ (f(u)\geq 0)$ for each $u\in S$, then $f(u)\leq 0~(f(u)\geq 0)$ for each $u\in X$.
\end{result}
\begin{lemma}\label{lem4.6}
 The set $C=\underset{n \in \mathbb{N}}{\bigcup}\left(\underset{1\leq i_1,\ldots,i_n\leq N}{\bigcup}L_{i_{1}\ldots i_{n}}\big(\left\{u_1,\ldots,u_N\right\}\big)\right)$ is dense in interval $I=[0,1]$, where $L_{i_{1}\ldots i_{n}}(u)=L_{i_{1}}(L_{i_{2}}(\ldots (L_{i_{n}}(u))))$ and $n \in \mathbb{N}.$
\end{lemma}
\begin{proof}
Let $u\in I$ be any point. Observe that for some $w\in \left\{u_1,\ldots,u_N\right\}$, we have $\lvert u-w \rvert\leq \underset{i\in J}{\max} \left\{\frac{u_i-u_{i-1}}{2}\right\}$. Since each $L_i$, is a contraction mapping with contraction coefficient $a_i$. Choose $a=\underset{i\in J}{\max} \{a_i\}$, then for each $u\in I$ and for each $\epsilon>0$ we can choose $w\in L_{i_{1}\ldots i_{n}}\big(\left\{u_1,\ldots,u_N\right\}\big)$ for some $n\in \mathbb{N}$ such that,
\[\lvert u-w \rvert\leq a^{n}\underset{i\in J}{\max} \left\{\frac{u_i-u_{i-1}}{2}\right\}<\epsilon.\]
This completes the proof.
\end{proof}

\begin{theorem}\label{Thm4.7}
Let $F,G \in \mathcal{C}(I , \mathcal{K}( \mathbb{R}))$ and $\Delta$ as defined in Theorem \ref{Thm3.2}, and $F(u_1),F(u_N), G(u_1)$, $G(u_N)$ are single-valued. If $F \le G$, then $F^{\alpha} \le G^{\alpha}$ provided $S_F, S_G \in \mathcal{C}(I , \mathcal{K}( \mathbb{R}))$ satisfying $S_F \le S_G$ and $S_F(u_1)=F(u_1),S_F(u_N)=F(u_N),S_G(u_1)=G(u_1),S_G(u_N)=G(u_N).$
\end{theorem}
\begin{proof}
Let $S_F, S_G \in \mathcal{C}(I , \mathcal{K}( \mathbb{R}) )$ such that $S_F \le S_G$ and $S_F(u_1)=F(u_1),~S_F(u_N)=F(u_N),~S_G(u_1)=G(u_1),~S_G(u_N)=G(u_N).$
Using Note \ref{Nt3.4}, we have 
\[F^{\alpha}(u_i)=F(u_i),~G^{\alpha}(u_i)=G(u_i) \text{ for each } i=1,\ldots,N.\]
From the self-referential equation, 
\[ F^{\alpha}\left(L_j(u)\right)= F\left(L_j(u)\right) + \alpha  \left[F^{\alpha}(u) - S_F(u)\right],\text{ and } G^{\alpha}\left(L_j(u)\right)= G\left(L_j(u)\right) + \alpha  \left[G^{\alpha}(u) - S_G(u)\right]\]
 for each $ u \in I_j, $ where $j \in J.$ For $u \in \Delta,$ we deduce \[F^{\alpha}\big(L_j(u)\big) \subset G^{\alpha}\big(L_j(u)\big) \text{ for any } j \in J.\]
 Applying the process repeatedly, we get
 \[F^{\alpha}\left(L_{i_{1}\ldots i_{n}}(u)\right)\subset G^{\alpha}\left(L_{i_{1}\ldots i_{n}}(u)\right) \text{ for any } i_1,\ldots, i_n \in J, ~ u \in \{u_1,\ldots,u_N\},\]
 where $L_{i_{1}\ldots i_{n}}(u)=L_{i_{1}}(L_{i_{2}}(\ldots (L_{i_{n}}(u))))$ and $n \in \mathbb{N}.$\\
 This implies that $F^{\alpha}(u)\subset G^{\alpha}(u)$ for each  $u\in \underset{n \in \mathbb{N}}{\cup}\left(\underset{1\leq i_1,\ldots,i_n\leq N}{\cup}L_{i_{1}\ldots i_{n}}\big(\left\{u_1,\ldots,u_N\right\}\big)\right).$ \\
Now using Lemma \ref{lem4.6} and Result \ref{res 4.5} we are done.
\end{proof}

\section{Dimensional Results}\label{sec5}
To move further in this section, we shall first observe some examples, then only one can understand the motivation behind this section.
\begin{example}
   Let $F_1:[0,1]\rightrightarrows \mathbb{R}$ is a set-valued map defined as $F_1(u)=\{0\}$, then according to \eqref{Gf1} graph of this function will be a line segment in $\mathbb{R}^2,$ and hence $\dim_H(G_{F_{1}})=1$.
\end{example}
\begin{example}\label{Ex5.2}
 Let $F_2:[0,1]\rightrightarrows \mathbb{R}$ is a set-valued map defined as $F_2(u)=[-1,1]$, then by \eqref{Gf1}, we have $G_{F_{2}}=[0,1]\times [-1,1],$ and hence $\dim_H(G_{F_{2}})=2$.
\end{example} 

\begin{example}
 Let $F_3:[0,1]\rightrightarrows \mathbb{R}$ be a set-valued map defined as $F_3(u)=C$, where $C$ is Cantor set. Then, by \eqref{Gf1} we have $G_{F_{3}}=[0,1]\times C,$ and hence $\dim_H(G_{F_{2}})=1+\frac{\log 2}{\log 3}$.
\end{example}
 Notice that $F_1, F_2,$ and $F_3$ are constant maps. Therefore, these are  Lipschitz and bounded variation maps as well. Unlike the case of a single-valued map, here we witness that the Hausdorff dimension of the graph of a set-valued Lipschitz map is other than 1, and the same observation holds for the graph of a set-valued bounded variation map also. In fact, one can always find a set-valued Lipschitz map or set-valued bounded variation map whose graph has dimension $\beta$ for any $1 \leq \beta \leq 2$. We observe that with the definition of the graph as in \eqref{Gf1}, we could not find any fascinating dimensional result, therefore we give a new definition of the graph of a set-valued map and study some dimensional results for this new definition of the graph.

 \begin{definition} \label{newg}
   Let $F:[0,1]\rightarrow \mathcal{K}(\mathbb{R})$ be a set-valued map, then a graph of $F$ is defined as;
   \begin{equation}\label{grph}
    \mathcal{G}(F)=\left\{(u,F(u)): F(u)\in \mathcal{K}(\mathbb{R}) \right\}\subset [0,1]\times \mathcal{K}(\mathbb{R}).  
   \end{equation} 
   Defined a metric on this graph, \[D_{\mathcal{G}}((u,F(u)),(w,F(w)))=\lvert u-w\rvert+H_d(F(u),F(w)).\]
   \end{definition}
Next, we prove the graph of $F^{\alpha}$ defined in \eqref{grph} is an attractor of an IFS defined on $I\times \mathcal{K}_c(\mathbb{R})$.

 Let us note the following lemma as a prelude. The motivation of this following lemma is coming from \cite[Proposition $1$]{barnsley2015bilinear}.
 \begin{lemma}\label{lemma5.1}
  Define a function $\mathfrak{d}:I\times \mathcal{K}_c(\mathbb{R}) \rightarrow [0, \infty)$ as
 \[\mathfrak{d}\big((u,A),(w,B)\big)=\lvert u-w \rvert+H_d\big(A+F^{\alpha}(w),B+F^{\alpha}(u)\big).\]
  Then, $I\times \mathcal{K}_c(\mathbb{R})$ with respect to $\mathfrak{d}$ is a complete metric space.
 \end{lemma}
 
 \begin{proof}
 Clearly, $\mathfrak{d}\big((u,A),(w,B)\big)=\mathfrak{d}\big((w,B),(u,A)\big)\geq 0$. Suppose that $\mathfrak{d}\big((u,A),(w,B)\big)=0$, then
 \begin{align*}
  &\lvert u-w \rvert+H_d\big(A+F^{\alpha}(w),B+F^{\alpha}(u)\big)=0\\
  \text{i.e.,}~ &\lvert u-w \rvert=0 \text{ and } H_d\big(A+F^{\alpha}(w),B+F^{\alpha}(u)\big)=0\\
  \text{i.e.,}~ & u=w ~H_d\big(A+F^{\alpha}(w),B+F^{\alpha}(u)\big)=H_d(A,B)=0\\
  \text{i.e.,}~& u=w \text{ and } A=B\\
  \text{i.e.,}~& (u,A)=(w,B).
 \end{align*}
 Now to prove that $\mathfrak{d}$ satisfies the triangle inequality. Take $(u_i, A_i)\in I \times \mathcal{K}_c(\mathbb{R})$ for $i=1,2,3.$ Then, we have
 \begin{align*}
     &\mathfrak{d}\big((u_1,A_1),(u_2,A_2)\big)\\
     =&\lvert u_1-u_2\rvert +H_d\big(A_1+F^{\alpha}(u_2),A_2+F^{\alpha}(u_1)\big)\\
     =& \lvert u_1-u_2\rvert +H_d\big(A_1+F^{\alpha}(u_2)+A_3+F^{\alpha}(u_3),A_2+F^{\alpha}(u_1)+A_3+F^{\alpha}(u_3)\big)\\
     \le& \big\{\lvert u_1-u_3\rvert +\lvert u_3-u_2\rvert \big\}+\big\{H_d\big(A_1+F^{\alpha}(u_3), A_3+F^{\alpha}(u_1)\big)+H_d\big(A_3+F^{\alpha}(u_2), A_2+F^{\alpha}(u_3)\big) \big\}.
     \end{align*}
     Hence,
     \[\mathfrak{d}\bigg((u_1,A_1),(u_2,A_2)\bigg)\leq \mathfrak{d}\bigg((u_1,A_1),(u_3,A_3)\bigg)+\mathfrak{d}\bigg((u_3,A_3),(u_2,A_2)\bigg).\]
     To prove completeness, let $\{(u_n,A_n)\}$ is a Cauchy sequence in $I\times \mathcal{K}_c(\mathbb{R})$. For $\epsilon>0$ there is an integer $N(\epsilon)$ such that
     \[\lvert u_n-u_m\rvert +H_d\big(A_n+F^{\alpha}(u_n), A_m+F^{\alpha}(u_m)\big)< \epsilon,~ \text{whenever}~ m,n\geq N(\epsilon).\]
    This shows $\{u_n\}$ is a Cauchy sequence of $I$, hence it converges to, say $u^*\in I.$ Since $F^{\alpha}$ is a uniformly continuous map, consequently  $\{F^{\alpha}(u_n)\}$ will also be a Cauchy sequence with respect to Hausdorff metric, and hence converges to $F^{\alpha}(u^*)\in \mathcal{K}_c(\mathbb{R})$. Then,
     \begin{align*}
         H_d(A_n,A_m)=& H_d\big(A_n+F^{\alpha}(u_n),A_m+F^{\alpha}(u_n)\big)\\
         \le & H_d\big(A_n+F^{\alpha}(u_n),A_n+F^{\alpha}(u_m)\big)+H_d\big(A_n+F^{\alpha}(u_m)+A_m+F^{\alpha}(u_n)\big)\\
         = & H_d\big(F^{\alpha}(u_n),F^{\alpha}(u_m)\big)+H_d\big(A_n+F^{\alpha}(u_m)+A_m+F^{\alpha}(u_n)\big)\\
        <& \frac{\epsilon}{2}+\frac{\epsilon}{2}=\epsilon.
     \end{align*}
     This implies, $\{A_n\}$ is a Cauchy sequence of $\mathcal{K}_c(\mathbb{R})$ and so it converges to, say $A^{*}\in \mathcal{K}_c(\mathbb{R}).$ Hence, $\{(u_n, A_n)\}$ converges to $(u^*, A^*) \in I\times \mathcal{K}_c(\mathbb{R}).$ This completes the proof.
     
 \end{proof}

 \begin{proposition}\label{prop5.6}
Define $W_{j}: I \times \mathcal{K}_c(\mathbb{R})\rightarrow I \times \mathcal{K}_c(\mathbb{R})$ for each $j\in J$ such that
 \[W_{j}(u,A)=\left(L_{j}(u), \alpha A+F(L_{j}(u))-\alpha S(u)\right).\] Then, each $W_j$ is a contraction map with respect to the metric defined in Lemma \ref{lemma5.1} , provided $\max\{\lvert \alpha \rvert,a_j\}<1$ for each $j\in J$.
 \end{proposition}
 
 \begin{proof}
 Let $(u,A),(w,B)\in I \times \mathcal{K}(\mathbb{R})$, then
\begin{align*}
      &\mathfrak{d}\big(W_j(u,A),W_j(w,B)\big)\\
     =&\mathfrak{d}\Big(\big(L_j(u),\alpha A+F(L_j(u))-\alpha S(u)\big),\big(L_j(w),\alpha B+F(L_j(w))+\alpha S(u)\big)\Big)\\
     =&\big\lvert L_j(u)-L_j(w)\big\rvert + H_d\Big(\alpha A+F(L_j(u))-\alpha S(u)+F^{\alpha}(L_j(w)),\\
     &\hspace{8cm}\alpha B+F(L_j(w))-\alpha S(w)+F^{\alpha}(L_j(u))\Big)\\
     =&\big\lvert L_j(u)-L_j(w)\big\rvert+ H_d\Big(\alpha A+F(L_j(u))-\alpha S(u)+F(L_j(w))+\alpha F^{\alpha}(w)-\alpha S(w),\\
     &\hspace{5.5cm}\alpha B+F(L_j(w))-\alpha S(w)+F(L_j(u))+\alpha F^{\alpha}(u)-\alpha S(u) \Big)\\
     =&a_j \big\lvert u-w \big\rvert+H_d\Big(\alpha A+\alpha F^{\alpha}(w),\alpha B+\alpha F^{\alpha}(u) \Big)\\
     =&a_j \big\lvert u-w \big\rvert+\lvert \alpha \rvert H_d\Big( A+ F^{\alpha}(w), B+ F^{\alpha}(u) \Big)\\
     \le& \max\{\lvert \alpha \rvert,a_j\} \Big(\big\lvert u-w \big\rvert+H_d\Big( A+ F^{\alpha}(w), B+ F^{\alpha}(u) \Big) \Big)\\
     =&\max\{\lvert \alpha \rvert,a_j\} \mathfrak{d}\big((u,A),(w,B)\big).
 \end{align*}
 Since $\max\{\lvert \alpha \rvert,a_j\}< 1$, each $W_j$ is a contraction mapping. 
 \end{proof}

 \begin{theorem}\label{Thm5.7}
 For each $j\in J$, let $W_{j}: I \times \mathcal{K}_c(\mathbb{R})\rightarrow I \times \mathcal{K}_c(\mathbb{R})$ be the map defined in Proposition \ref{prop5.6}.
 Then, by Definition \ref{newg}, the graph of $F^{\alpha}$ will be an attractor of the \setword{IFS}{Word:IFS}, $\left\{\left(I\times \mathcal{K}_c(\mathbb{R}),\mathfrak{d}\right);W_1,\ldots, W_{N-1}\right\}.$
 \end{theorem}
 
 \begin{proof}
 Since $I=\underset{j\in J}{\bigcup}L_{j}(I)$. Then, from \eqref{self12}, we have
 \begin{align*}
   \underset{j\in J}{\bigcup}W_{j}(\mathcal{G}(F^{\alpha})) &= \underset{j\in J}{\bigcup}\big\{W_{j}(u,F^{\alpha}(u)):u\in I\big\}\\
   &=\underset{j\in J}{\bigcup}\big\{\big(L_{j}(u), \alpha F^{\alpha}(u)+F(L_{j}(u))-\alpha S(u)\big): u\in I\big\}\\
   &= \underset{j\in J}{\bigcup}\big\{\left(L_{j}(u),F^{\alpha}(L_{j}(u))\right): u\in I\big\}\\
   &= \underset{j\in J}{\bigcup}\big\{(u, F^{\alpha}(u)): u\in L_{j}(I)\big\}\\
   &= \mathcal{G}(F^{\alpha}).
 \end{align*}
 This completes the proof.
 \end{proof}

 Schief \cite{schief1996self} noted that the dimensional results for Euclidean spaces do not have simple generalizations to complete metric spaces. Following his work, Nussbaum et al. \cite{nussbaum2012positive} proved a more general result in the setting of a complete metric space.
Answering a question raised in \cite{nussbaum2012positive}, Verma \cite{verma2021hausdorff} has shown the Hausdorff dimension of the invariant set under the SOSC. He explores several dimensional aspects of sets in complete metric space. In his book \cite{falconer2004fractal}, Falconer studied the dimensional results of sets in Euclidean spaces. Given \cite{verma2021hausdorff}, we may assure the reader that some results, which we will use, also hold in a general complete metric space.

 \begin{theorem} \label{Thm5.8}
 Let $\mathcal{I}=\{I\times \mathcal{K}_c(\mathbb{R}); W_1,\ldots W_{N-1}\}$ be the \ref{Word:IFS} defined in Theorem \ref{Thm5.7} such that
 \[r_iD_{\mathcal{G}}\big((u,A),(w,B)\big) \leq D_{\mathcal{G}}\big(W_i(u,A),W_i(w,B)\big) \leq R_i D_{\mathcal{G}}\big((u,A),(w,B)\big)\]
 for every $(u,A),(w,B)\in I \times \mathcal{K}_c(\mathbb{R}),$ where $0<r_i \leq R_i <1$ for all $i\in J.$ Then, $t_*\leq \dim_H(\mathcal{G}(F^{\alpha}))\leq t^*,$ where $t_*$ and $t^*$ are characterized by $\sum\limits_{i=1}^{N}r_{i}^{t^*}=1$ and $\sum\limits_{i=1}^{N}R_{i}^{t^*}=1,$ respectively.
 \end{theorem}
 
 \begin{proof}
 For purposed upper bound one can refer \cite[Proposition 9.6]{falconer2004fractal}( see also, \cite[Theorem $2.12$]{verma2021hausdorff}). For lower bound of Hausdorff dimension of $\mathcal{G}(F^{\alpha})$ we progress as follows.\\
Set $V=(u_1,u_N)\times \mathcal{K}_c(\mathbb{R})$, an open set in $I\times \mathcal{K}_c(\mathbb{R})$. Since for each $i,j \in J ~\text{with}~ i\neq j,$ we have
 \[L_j\big((u_1,u_N)\big)=(u_j, u_{j+1}) \text{ and } L_i\big((u_1,u_N)\big)\cap L_j\big((u_1,u_N)\big)=\emptyset,\]
 hence for each $i,j\in J \text{ and } i\neq j$, we have
 \[W_j(V)=(u_j, u_{j+1})\times \mathcal{K}_c(\mathbb{R}) \text{ and } W_i(V)\cap W_j(V)=\emptyset.\]
 Therefore,
\[\bigcup_{i=1}^{N-1}W_i(V)=\bigcup_{i=1}^{N-1}\big\{(u_i,u_{i+1})\times \mathcal{K}_c(\mathbb{R})\big\}\subseteq V~\text{and}~W_i(V)\cap W_j(V)=\emptyset.\]
Then, by using Definition \ref{OSC} IFS satisfies OSC. We have $V\cap \mathcal{G}(F^{\alpha})\neq \emptyset$ this implies that IFS is satisfying SOSC. Since $V\cap \mathcal{G}(F^{\alpha})\neq \emptyset$, we have an $i\in J^*$ such that $\mathcal{G}(F^{\alpha})_i\subset V,$ where $J^{*}=\underset{m\in \mathbb{N}}{\cup}\{1,\ldots N-1\}^{m},$ collection of all finite sequences whose terms are in $J$ and 
\[\mathcal{G}(F^{\alpha})_i=W_i(\mathcal{G}(F^{\alpha})):= W_{i_{1}}\circ W_{i_{2}} \circ \cdots \circ W_{i_{m}}(\mathcal{G}(F^{\alpha}))\]
for $i\in J^m=J\times \cdots \times J~(\text{m-times})$ and $m\in \mathbb{N}$. Observe that for any $j\in J^m$ and $k\in \mathbb{N}$, the sets, $\mathcal{G}(F^{\alpha})_{j_i}$, are disjoint. Further, the IFS $\{W_{j_{i}}:j\in J^k\}$ satisfies the hypothesis of \cite[ Proposition 9.7]{falconer2004fractal} ( see also, \cite[Theorem $2.35$]{verma2021hausdorff}). Therefore, with the notation $r_j=r_{j_{1}}r_{j_{2}}\cdots r_{j_{k}}$ we have $t_k\leq \dim_H(G^*)$, where $G^*$ is an attractor of the \ref{Word:IFS} and $\sum\limits_{j\in J^k}r^{t_k}_{j_{i}}=1.$ Since $G^*\subset \mathcal{G}(F^{\alpha})$, $t_k\leq \dim_H(G^*) \leq \dim_H(\mathcal{G}(F^{\alpha})).$ Let if possible $ \dim_H(\mathcal{G}(F^{\alpha}))< t_*$, where $\sum\limits_{i=1}^{N}r_{i}^{t_*}=1.$ Then, $t_k < t_*.$ Now, we have
\begin{align*}
    r_{i}^{-t_k}=&\sum\limits_{j\in J^k}r_{j}^{t_k}
    \geq  \sum\limits_{j\in J^k}r_{j}^{\dim_H(\mathcal{G}(F^{\alpha}))}
    =  \sum\limits_{j\in J^k}r_{j}^{t_*}r_{j}^{\dim_H(\mathcal{G}(F^{\alpha}))-t_*}\\
    \geq & \sum\limits_{j\in J^k}r_{j}^{t_*}r_{max}^{k(\dim_H(\mathcal{G}(F^{\alpha}))-t_*)}\\
    =& r_{max}^{k(\dim_H(\mathcal{G}(F^{\alpha}))-t_*)},
\end{align*}
where $r_{\max}=\max\{r_1, r_2,\ldots , r_N\}.$ Since $r_{\max}<1$, $r_{max}^{k(\dim_H(\mathcal{G}(F^{\alpha}))-t_*)}$ tends to infinity as $k$ tends to infinity, and therefore $r_{i}^{-t_k}$ is unbounded, which is a contradiction. Hence, $\dim_H(\mathcal{G}(F^{\alpha}))\geq t_*$, which is the required result.
 \end{proof}
 The following theorem is an immediate application of the Theorem \ref{Thm5.8}.
 
 \begin{theorem}
 Consider $F:I\rightarrow \mathcal{K}_c(\mathbb{R})$ is a set-valued map. If $\lvert \alpha \rvert< \min \{a_i: i\in J\}$, then $\dim_H(\mathcal{G}(F^{\alpha}))=1.$
 \end{theorem}
 \begin{proof} 
 Using Proposition \ref{prop5.6} for every pair $(u,A),(w,B)\in I\times \mathcal{K}_c(\mathbb{R})$, we have
 \[D_{\mathcal{G}}(W_i(u,A),W_i(w,B))\leq a_iD_{\mathcal{G}}((u,A),(w,B)) \text{ for } i\in J.\]
 Since $\sum\limits_{i=1}^{N-1}a_i=1$, then by Theorem \ref{Thm5.8}, $\dim_H(\mathcal{G}(F^{\alpha}))\leq 1$. This concludes the proof.
 
 \end{proof}
 
   \begin{theorem}\label{Thm5.9}
   If $F:[0,1]\rightarrow \mathcal{K}(\mathbb{R})$ is a set-valued Lipschitz map having Lipschitz constant $l$ and the graph of $F$ is as defined in \eqref{grph}, then $\dim_H(\mathcal{G}(F))=1.$
   \end{theorem}
   \begin{proof}
   To proof this theorem, it will be sufficient to define a bi-Lipschitz map between $[0,1]$ and $\mathcal{G}(F)$. 
   Define $T:[0,1]\rightarrow \mathcal{G}(F)$ such that $T(u)=(u,F(u))$. Then, we have
   \begin{align}
       D_\mathcal{G}(Tu,Tw)&=D_{\mathcal{G}}((u,F(u)),(w,F(w))) \nonumber\\
       &= \lvert u-w \rvert+ H_d(Fu,Fw)\nonumber\\
       &\leq \lvert u-w \rvert+l\lvert u-w \rvert\nonumber\\
       &\leq (1+l) \lvert u-w \rvert\nonumber,\\
       \text{that is, } D_{\mathcal{G}}(Tu,Tw)&\leq (1+l)\lvert u-w \rvert \label{6.2}
   \end{align}
    and 
    \begin{align}
        D_{\mathcal{G}}(Tu,Tw)&=D_{\mathcal{G}}((u,Fu),(w,Fw))\nonumber\\
        &= \lvert u-w \rvert +H_d(Fu,Fw)\nonumber\\
        \text{that is, }D_{\mathcal{G}}(Tu,Tw)&\geq \frac{1}{2}\lvert u-w \rvert. \label{6.3}
    \end{align}
Equations \eqref{6.2} and \eqref{6.3} will prove the bi-Lipschitz nature of $T$.\\ 
Hence, $\dim_H(\mathcal{G}(F))=1.$   
   \end{proof}
  
\begin{theorem}
Let $F, S \in \mathcal{C}(I, \mathcal{K}_c(\mathbb{R}))$ are Lipschitz functions such that $S(u_1)-F(u_1)=S(u_N)-F(u_N),$ and let $\alpha \in (-1,1)$. Then, $\dim_H(\mathcal{G}(F^{\alpha})) =1$ provided that $  \lvert \alpha \rvert<a:= \min\{a_j: j \in J \} .$
\end{theorem}
   \begin{proof}
In view of Theorem \ref{Thm5.9} and Theorem \ref{Thm3.5}, the proof follows, hence we omit.   
   \end{proof}
   \begin{lemma}\label{lipdim}
   Let $F,T:[0,1]\rightarrow \mathcal{K}(\mathbb{R})$ be set-valued Lipschitz map with Lipschitz constant $l$, then $\dim_H(\mathcal{G}(F+T)) = \dim_H(\mathcal{G}(T))$, where $(F+T)(u):=F(u)+T(u)$ and $F(u)+T(u)$ denotes the Minkowski sum of $F(u)$ and $T(u)$.
   \end{lemma}
   \begin{proof}
   To establish the proof of this lemma, it will be sufficient to show the existence of a Lipschitz map from $\mathcal{G}(T)$ to $\mathcal{G}(F+T)$.
   Define $\Phi : \mathcal{G}(T)\rightarrow \mathcal{G}(F+T)$ such that $\Phi (u,T(u))=(u, F(u)+T(u))$. It is easy to see that $\Phi$ is well defined and onto. Now to get its Lipschitz behaviour, we have
   \begin{align*}
       D_{\mathcal{G}}(\Phi (u,T(u)),\Phi (w,T(w)))&=D_{\mathcal{G}}((u,F(u)+T(u)),(w,F(w)+T(w)))\\
       &= \lvert u-w \rvert + H_d(F(u)+T(u), F(w)+T(w))\\
       &\leq \lvert u-w \rvert + H_d(F(u),F(w))+ H_d(T(u),T(w))\\
       &\leq \lvert u-w \rvert + l \lvert u-w \rvert +H_d(T(u),T(w))\\
       &\leq (1+l) \left \{\lvert u-w \rvert + H_d(T(u),T(w))\right\}.
   \end{align*}
   That is, $D_{\mathcal{G}}(\Phi (u,T(u)),\Phi (w,T(w))) \leq (1+l) D_{\mathcal{G}}((u,T(u)),(w,T(w))).$
   Hence, $\Phi $ being a Lipschitz map implying that 
   \[ \dim_H(\mathcal{G}(F+T)) \le \dim_H(\mathcal{G}(T)) \text{ and }  \dim_B(\mathcal{G}(F+T)) \le \dim_B(\mathcal{G}(T)) .\] \\
   For other side of the inequality, let $t> \dim_H(\mathcal{G}(F+T)).$ Then, by definition of the Hausdorff dimension, we have the following.\\
   For each $\epsilon>0$ and for each $\eta>0,$ there is an open cover $\{U_n:n \in \mathbb{N}\} $ of $\mathcal{G}(F+T))$ such that $\lvert U_n\rvert  < \eta$ and $\sum_{n \in \mathbb{N}} \lvert U_n\rvert^s < \epsilon.$
   Note that 
   \begin{equation}\label{sept26}
   \begin{aligned}
       \mathcal{G}(T)=& \{(u,T(u)): u \in I\}\\ \subseteq & \{ (u, T(u)+F(u)-F(u)): u \in I\} \\ \subseteq & \{(u,T(u)+F(u)): u \in I\}+ \{ (0, -F(u)):u \in I\}\\ = & \mathcal{G}(F+T))+ \{ (0, -F(u)):u \in I\}.
       \end{aligned}
       \end{equation}
    
   Define $V_n=U_n+ \{(0,-F(u)):u \in I \text{ such that } (u,T(u)+F(u)) \in U_n\}.$ Observe that each $V_n$ is open and $\mathcal{G}(T) \subseteq \underset{n \in \mathbb{N}}{\cup} V_n.$ Further, $\lvert V_n\rvert  \le (1+l)\lvert U_n\rvert $ and $\lvert V_n\rvert < (1+l)\eta.$ Then, 
   \[ \sum_{n \in \mathbb{N}} \lvert V_n\rvert^t \le (1+l)^t \sum_{n \in \mathbb{N}} \lvert U_n\rvert^t \le  (1+l)^t\epsilon.
   \]
   This gives  $\mathcal{H}_{\eta}^t(\mathcal{G}(T))=0,$ that is, the $t$-dimensional Hausdorff measure, $\mathcal{H}^t(\mathcal{G}(T))=0.$ Therefore, we have $\dim_H(\mathcal{G}(T))\le \dim_H(\mathcal{G}(F+T)) ,$ proving $\dim_H(\mathcal{G}(T))= \dim_H(\mathcal{G}(F+T)) .$ 
   Next, using \eqref{sept26}, we obtain
\begin{align*}
  \overline{\dim}_B(\mathcal{G}(T))&=\varlimsup_{\delta \rightarrow 0} \frac{\log N_{(1+l) \eta}(\mathcal{G}(T))}{- \log ((1+l) \eta)}\\
  &\le \varlimsup_{\eta \rightarrow 0} \frac{\log N_{ \eta}(\mathcal{G}(F+T))}{- \log ((1+l) \eta)}= \varlimsup_{\eta \rightarrow 0} \frac{\log N_{\eta}(\mathcal{G}(F+T))}{- \log( \eta)}=\overline{\dim}_B(\mathcal{G}(F+T)),
\end{align*}
as desired.
    \end{proof}
    
    \begin{remark}
The above lemma holds for single-valued maps also (see for instance \cite[Lemma $3.2$]{verma2020bivariate}), but proof of this is neither straight forward nor just a simple extension of single-valued map. Because Hausdorff metric does not satisfy the parallelogram law while in \cite[Lemma $3.2$]{verma2020bivariate} metric is usual metric defined on $\mathbb{R}^{n}$ which satisfies the parallelogram law and gives the privilege to enjoy the bi-Lipschitz property to $T$ defined in \cite[Lemma $3.2$]{verma2020bivariate} ($\Phi$ in our case).
    \end{remark}
    
    In view of the Lipschitz invariance property of dimension, one may conclude that the upcoming theorem holds for all aforementioned dimensions.

\begin{theorem}\label{densethm}
Consider $1 \leq \beta $. Then, set  $\mathcal{S}_{\beta}:=\{F\in  \mathcal{C}(I,\mathcal{K}(\mathbb{R})): \dim_H (\mathcal{G}(F)) = \beta\}$ is dense in $\mathcal{C}(I,\mathcal{K}(\mathbb{R})).$
\end{theorem}
\begin{proof}
   Let $F\in \mathcal{C}(I,\mathcal{K}(\mathbb{R}))$ and $\epsilon>0.$ Using the density of $\mathcal{L}ip(I,\mathcal{K}(\mathbb{R}))$ in $\mathcal{C}(I,\mathcal{K}(\mathbb{R}))$, there exists $G$ in $\mathcal{L}ip (I,\mathcal{K}(\mathbb{R}))$ such that $$ \lVert F-G \rVert _{\infty} < \frac{\epsilon}{2}.$$ Further, we consider a non-vanishing function $H \in \mathcal{S}_{\beta}.$ Let $H_*= G +\frac{\epsilon}{2\lVert H\rVert _{\infty}}H,$ which immediately gives 
   \[\lVert G-H_*\rVert _{\infty} \le  \frac{\epsilon}{2}.\]
   This together with Lemma \ref{lipdim} implies that $\dim(Gr(H_*))=\dim(Gr(H))=\beta.$ Hence, we have $H_* \in \mathcal{S}_{\beta}$ and $$ \lVert F-H_*\rVert _{\infty} \le  \lVert F-G \rVert _{\infty} + \lVert G-H_*\rVert _{\infty} < \epsilon .$$ This completes the proof.
\end{proof}   
    
    Before proving our next result, let us note the following lemma as a prelude.
    \begin{lemma}\label{new921}
   Consider $A,B,C$ are compact subsets of $\mathbb{R}.$ Then, $$H_d(AB,CB) \le \sup_{b \in B} \lvert b \rvert H_d(A,C),$$ 
    where $YZ=\left\{yz: y\in Y\in \mathcal{K}(\mathbb{R}),~ z\in Z\in \mathcal{K}(\mathbb{R})\right\}.$
    \end{lemma}
    \begin{proof}
    We have 
    \begin{align*}
    H_d(AB,CB)& = \max\Big\{ \sup_{ab \in AB} \inf_{cb' \in CB} \lvert ab -cb'\rvert , \sup_{cb' \in CB} \inf_{ab \in AB} \lvert cb'-ab \rvert \Big\}\\
    & \le \max\Big\{ \sup_{ab \in AB} \inf_{cb \in Cb} \lvert ab -cb\rvert , \sup_{cb' \in CB} \inf_{ab' \in Ab'} \lvert cb'-ab'\rvert \Big\}\\
    & \le \max\Big\{ \sup_{ab \in AB} \inf_{cb \in Cb}\lvert b \rvert \lvert a -c\rvert , \sup_{cb' \in CB} \inf_{ab' \in Ab'}\lvert b' \rvert \lvert c-a \rvert\Big\}\\
    & \le \max\Big\{ \sup_{a \in A, b \in B}(\lvert b \rvert \inf_{cb' \in CB} \lvert a -c\rvert ), \sup_{c \in C, b' \in B} (\lvert b' \rvert \inf_{ab' \in Ab'} \lvert c-a \rvert)\Big\}\\
    & \le \sup_{ b \in B}\lvert b \rvert  \max\Big\{ \sup_{a \in A} \inf_{c \in C} \lvert a -c\rvert , \sup_{c \in C}  \inf_{a \in A} \lvert c-a \rvert\Big\} \\
    & = \sup_{ b \in B}\lvert b \rvert H_d(A,C),
    \end{align*}
    proving the assertion.
    \end{proof}
    
    Next we define the multiplication of set-valued maps $F,L:W \subseteq \mathbb{R} \rightrightarrows \mathbb{R}$ by $(FT)(w)=F(w)T(w).$
    \begin{lemma}\label{Lem5.14}
     Consider $F ,T:[0,1] \rightarrow \mathcal{K}(\mathbb{R})$ to be set-valued  Lipschitz map with Lipschitz constant $l$. Then, 
     $$ \dim_H (\mathcal{G}(FT)) \le \dim_H (\mathcal{G}(T)).$$
    \end{lemma}
     \begin{proof}
     Define $\Phi  : \mathcal{G}(T) \rightarrow \mathcal{G}(FT) $ such that
     \[\Phi \big((u,T(u))\big)=\big(u,F(u)T(u)\big).\]
     Choose $M=\max \{1+l \underset{z\in \underset{u\in[0,1]}{\cup} Tu}{\sup}\lvert z \rvert, \underset{v\in \underset{w\in[0,1]}{\cup} Fw}{\sup}\lvert v \rvert \}.$\\
     Notice that $\Phi$ is well defined and surjective. To prove our lemma, it is enough to prove $\Phi$ is Lipschitz map. For this,
     \begin{align*}
          D_{\mathcal{G}}(\Phi (u,Tu),\Phi (w,Tw))&=D_{\mathcal{G}}((u,FuTu),(w,FwTw))\\
       &= \lvert u-w \rvert + H_d(FuTu, FwTw)\\
       &\leq \lvert u-w \rvert +H_d(FuTu,FwTu)+H_d(FwTu,FwTw)\\
       &\leq \lvert u-w \rvert +\underset{z\in Tu}{\sup}\lvert z\rvert H_d(Fu,Fw)+\underset{v\in Fw}{\sup}\lvert v\rvert H_d(Tu,Tw)\\
       &\leq \lvert u-w \rvert +\underset{z\in Tu}{\sup}\lvert z\rvert l \lvert u-w \rvert+\underset{v\in Fw}{\sup}\lvert v\rvert H_d(Tu,Tw)\\
       &\leq M \left \{\lvert u-w \rvert + H_d(Tu, Tw)\right\}.
       \end{align*}
       Hence, $D_{\mathcal{G}}(\Phi (u,Tu),\Phi (w,Tw)) \leq M D_{\mathcal{G}}((u,Tu),(w,Tw)).$\\
This completes the proof.
     \end{proof}

     \begin{remark}
       In the Lemma \ref{Lem5.14}, equality may not hold in general. For instance, consider $T$ be a Weierstrass function whose Hausdorff dimension is strictly greater than 1 (refer \cite{shen2018hausdorff}) and $F$ to be the zero function. Then, we obtain $1=\dim_H(\mathcal{G}(FT)) < \dim_H (\mathcal{G}(T)).$
     \end{remark}

   \begin{definition}
  Consider $W$ be a bounded and closed interval of $\mathbb{R}$ and $F: W \rightrightarrows  \mathbb{R}$ is a set-valued map. The maximum range of $F$ over the rectangle $W$ is defined as
   $$ R_F[W]= \sup_{x,y \in W} ~\sup_{w,z \in F(x)\cup F(y)}\lvert w-z\rvert .$$
   \end{definition}
   As indicated in the introductory section, next we shall provide a set-valued analogue of \cite[Proposition $11.1$]{falconer2004fractal}.
   \begin{proposition} \label{BBVL2}
   Assume $F :[w,u]  \rightrightarrows \mathbb{R}$ be a set-valued continuous map, $ 0 < \eta < u-w,$ and $ \frac{u-w}{\eta} \le m \le 1+ \frac{u-w}{\eta }$ for some $m  \in \mathbb{N}.$  If $N_{\eta}(G_F)$ is the number of $\eta$-boxes that intersect the graph of $F,$  then 
\[ \frac{1}{\eta}  \sum_{i=1}^{m} R_F[W_{i}] \le N_{\eta}(G_F) \le 2m + \frac{1}{\eta}  \sum_{i=1}^{m} R_F[W_{i}],\]
where $W_i=[i\eta,(i+1)\eta]$.
   \end{proposition}
   \begin{proof}
   The count of squares having $\eta$ side length in the part above $W_{i}$ intersecting the graph of $F$ is at least $\frac{R_F[W_{i}]}{ \eta} $ and at most  $2+ \frac{R_F[W_{i}]}{ \eta},$ using the continuity of $F$. Taking Sum over all such parts yield the required bounds.
   \end{proof}
   
\begin{example}
Consider $F:[0,1]\rightrightarrows \mathbb{R}$ is a set-valued map defined as $F(x)=[-1,1].$ By Proposition \ref{BBVL2}, we have
 \[\overline{\dim}_B(G_F)=\varlimsup_{\eta \rightarrow 0} \frac{\log N_{ \eta}(G_F)}{- \log (\eta)} \le \varlimsup_{\eta \rightarrow 0} \frac{\log \Big(2m + \frac{1}{\eta}  \overset{m}{\underset{i=1}{\sum}} R_F[W_{i}]\Big)}{- \log (\eta)}\le \varlimsup_{\eta \rightarrow 0} \frac{\log \Big(2m + \frac{1}{\eta}  \overset{m}{\underset{i=1}{\sum}} 2 \Big)}{- \log (\eta)}=2,\]
 because $R_F[W_{i}]=2 \text{ for each } i=1,\ldots, m$ and $W_i=[i\eta, (i+1)\eta]$. Similarly, 
 $$
 \underline{\dim}_B(G_F)=\varliminf_{\eta \rightarrow 0} \frac{\log N_{ \eta}(G_F)}{- \log (\eta)} \ge \varliminf_{\eta \rightarrow 0} \frac{\log \Big( \frac{1}{\eta}  \overset{m}{\underset{i=1}{\sum}} R_F[W_{i}]\Big)}{- \log (\eta)}= \varliminf_{\eta \rightarrow 0} \frac{\log \Big( \frac{1}{\eta}  \overset{m}{\underset{i=1}{\sum}} 2 \Big)}{- \log (\eta)}=2.$$ Therefore, $\dim_B(G_F)=2.$ This shows that Proposition \ref{BBVL2} will be very useful in estimating or finding box dimension of set-valued functions.
\end{example}  
  
\section{Conclusion and Future Direction}\label{conclusion}

In this paper, the term $\alpha$-fractal function has been introduced (Theorem \ref{Thm3.2}), corresponding to set-valued maps. Next, we noticed that, unlike a single-valued $\alpha$-fractal function, a set-valued $\alpha$-fractal function is generally not interpolatory. Still, under certain conditions, it is interpolatory in nature (Remark \ref{rmrk3.3}, Note \ref{Nt3.4}). Also, some properties of this fractal function have been observed (Theorem \ref{Thm3.5}, Theorem \ref{bounded2}). After that, the existence of fractal polynomial, which approximates the convex set-valued map, was established (Theorem \ref{BFOTHM5}). Also, the concept of constrained approximation for set-valued maps is introduced (Theorem \ref{Thm4.7}). Further, we added the definition of a graph of a set-valued map (Definition \ref{newg}) and calculated the fractal dimension of this graph for some class of set-valued maps (Theorem \ref{Thm5.9}, Lemma \ref{lipdim}, and Lemma \ref{Lem5.14}).\\

In this paper, we have taken Minkowski sum of two sets. In the future, we may try to study the fractal functions using metric linear sum of two sets introduced by Dyn and her group \cite{berdysheva2019metric}. Here most of the results are available for convex set-valued maps but using metric linear sum of two sets, we may try to establish these results for the compact set-valued map.\\

Further, fractional calculus for the single-valued map has been widely explored. See, for instance \cite{liang2010box,chandra2021calculus}. In the future, we may try to extend this concept of fractional calculus for set-valued maps and estimate some dimensional results for the graph of the fractional integral and fractional differentiation of set-valued maps.\\

Another future direction of work is in selection of set-valued maps. Following remark can work as a motivation.
\begin{remark}
Consider $F : I \to \mathcal{K}(\mathbb{R})$ to be set-valued function, then a
function $f : I \to \mathbb{R}$ will be characterized as a selection of $F$ if $f(x) \in F(x)$ for
all $x \in I.$ It is an interesting fact to note that for any $1 \le \beta \le 2$, we are getting a selection  $f_{\beta} : I \to \mathbb{R}$ of map $F_2$ of Example \ref{Ex5.2} such that $\dim_H(Gr(f_{\beta}))=\beta.$ This motivates us to ask a natural question that whether such a selection respecting dimension exists or not.
\end{remark}

\section*{Declaration}

\textbf{Conflicts of interest.} We do not have any conflict of interest.\\
\\
\noindent
\textbf{Data availability:} No data were used to support this study.\\
\\
\noindent
\textbf{Code availability:} Not applicable\\
\\
\noindent
\textbf{Authors' contributions:} Each author contributed equally to this manuscript.

\section*{Acknowledgment}
This work is supported by MHRD Fellowship to the 1st author as TA-ship at the Indian Institute of Technology (BHU), Varanasi.

Some results of this paper have been presented at the conference, `` AMS Fall Western Virtual Sectional Meeting (formerly at University of New Mexico) : SS 13A - Special Session on Fractal Geometry and Dynamical Systems. , October 23-24, 2021".

\bibliographystyle{abbrv}
\bibliography{References}

\end{document}